\documentclass[11pt]{article}

\usepackage{amsmath,amssymb}
\usepackage{QED}
\usepackage[all]{xy}
\usepackage{pgf}
\usepackage{tikz}

\newtheorem{definition}{Definition}
\newtheorem{proposition}[definition]{Proposition}
\newtheorem{lemma}[definition]{Lemma}
\newtheorem{theorem}[definition]{Theorem}
\newtheorem{corollary}[definition]{Corollary}
\newtheorem{remark}[definition]{Remark}

\usepackage{anysize}
\marginsize{3cm}{3cm}{3cm}{3cm}

\title{From winning strategy to Nash equilibrium}

\author{St\'ephane Le Roux}

\begin{document}
\maketitle

\begin{abstract}
Game theory is usually considered applied mathematics, but a few game-theoretic results, such as Borel determinacy, were developed by mathematicians for mathematics in a broad sense. These results usually state determinacy, \textit{i.e.} the existence of a winning strategy in games that involve two players and two outcomes saying who wins. In a multi-outcome setting, the notion of winning strategy is irrelevant yet usually replaced faithfully with the notion of (pure) Nash equilibrium. This article shows that every determinacy result over an arbitrary game structure, \textit{e.g.} a tree, is transferable into existence of multi-outcome (pure) Nash equilibrium over the same game structure. The equilibrium-transfer theorem requires cardinal or order-theoretic conditions on the strategy sets and the preferences, respectively, whereas counter-examples show that every requirement is relevant, albeit possibly improvable. When the outcomes are finitely many, the proof provides an algorithm computing a Nash equilibrium without significant complexity loss compared to the two-outcome case. As examples of application, this article generalises Borel determinacy, positional determinacy of parity games, and finite-memory determinacy of Muller games.
\end{abstract}


\section{Introduction}\label{sect:intro}

Game theory is the theory of competitive interactions between decision makers having different interests. Its primary purpose is to further understand such real-world interactions through mathematical modelling, so it usually studies games that involve many players and many possible outcomes, which are meant to describe faithfully the many shades of situations involving many stakeholders. Apart from some earlier related works, the field of game theory is usually said to be born in the first part of the 20th century, especially thanks to von Neumann~\cite{NM44}, but also Borel~\cite{Borel21} and some others. Since then, it has been applied to many concrete areas such as economics, political science, evolutionary biology, computer science, \textit{etc}. Conversely, specific problems in these concrete areas have been triggering new general questions and have thus been helpful in developing game theory. 

Surprisingly, game theory has also provided useful point of view and terminology to abstract areas such as logic, descriptive set theory, and theoretical informatics. For instance, Martin~\cite{Martin90} describes the (quasi-)Borel sets \textit{via} the guaranteed existence of winning strategies in some games that are built using these (quasi-)Borel sets. More specifically, these games involve only two players and two outcomes saying who wins, \textit{i.e.} they are win-lose games, and they are built on infinite trees. Similarly, \cite{BL69}, \cite{GH82},\cite{EJ91}, \cite{Mostowski91}, \textit{etc}, relate some statements about logic to the existence/computation of simple winning strategies in some two-player win-lose games that are built on finite graphs, namely Muller games and parity games. In such frameworks, existence of a winning strategy (of some sort) is called determinacy, and the games that enjoy it are said to be determined.

There are substantial differences between the two above-mentioned types of game theory: games for logic are simple in terms of players, outcomes, and preferences, but the underlying structures are complex (since they pertain to the objects under study), whereas it is the other way round with games for economics, where the underlying structures involve, \textit{e.g.}, finiteness and continuous functions. Of course one would wish to get the best of the two worlds and to understand multi-player, multi-outcome games that are built on complex structures.

The above wish may partly come true. Indeed for some games of specific structures, determinacy results have already been generalised by considering the same structures for many player, many outcomes and the usual generalisation of the notion of winning strategy, namely the notion of Nash equilibrium: for instance, quasi-Borel determinacy was generalised in~\cite{SLR13} for (infinitely) many players and (infinitely) many outcomes; similarly, finite-memory determinacy of Muller games was generalised in~\cite{PS09}; and a similar question was asked in~\cite{Ummels05} for parity games. Said otherwise, for specific structures such as infinite trees or finite graphs, existence of winning strategies for all two-player win-lose games that are built on these structures might be transferred, through ad-hoc proofs, into existence of Nash equilibrium for all multi-player multi-outcome games that are built on these same structures. 

This article shows through a single, uniform proof that such a transfer of equilibrium from win-lose to multi-outcome setting holds irrespective of the structure of the game, \textit{i.e.} it need not be a tree, or a graph, \textit{etc.} However, a universal equilibrium-transfer theorem seems to hold for two-player games only, and to fail for three players already, although a clear-cut general counterexample is still missing. Furthermore, the general order-theoretic condition of the theorem, \textit{i.e.} that the preference orders over the outcomes be of (uniformly) finite height, can be relaxed to "inverted" well-foundness when strategy sets are countable. The case in between, where only one strategy set is countable and the inverse relation of preferences are well-founded yet with chains of arbitrary length, is still open.

Note that the Nash equilibria that are considered in this article are all pure since the concept need not be weakened through probabilistic means.

Section~\ref{sect:scgc} introduces the main result of this article intuitively, using an example; Section~\ref{sect:def} defines the required concepts and makes two straightforward remarks; Section~\ref{sect:mr-psc} states the main result of this article, \textit{i.e.} the equilibrium-transfer theorem, and provides a proof in a very simple yet informative case. Section~\ref{sect:tt} presents the equilibrium-transfer theorem and details the algorithmic content of the proof when the outcomes are finitely many; Section~\ref{sect:app-et} invokes the equilibrium-transfer theorem to generalise Martin's theorem on Borel determinacy, positional determinacy of parity games (with infinitely many priorities), and finite-memory determinacy of Muller games; Section~\ref{sect:ltt} gives counterexamples to reasonable candidates to generalise the theorem; and Section~\ref{sect:c} concludes and shows in passing that, generally speaking in game theory, linearly ordered preferences do not account for partially ordered preferences.

\subsection{From simple example to general idea}\label{sect:scgc}

Let us exemplify what the main result of this article, \textit{i.e.} the equilibrium-transfer theorem, can actually do. A finite real-valued two-player game in extensive form is an object built on a finite rooted tree; each internal node is owned by exactly one player, and each leaf encloses a real-valued payoff function, \textit{i.e.} an ordered pair assigning one real number to each player. The leftmost two objects below are such games, the second game being even a win-lose game. Intuitively, the left-most game is played as follows: player $b$ at the root chooses left or right. Right yields payoff $1$ for $a$ and $0$ for $b$, left requires a choice from player $a$, and so on until a leaf is reached. 

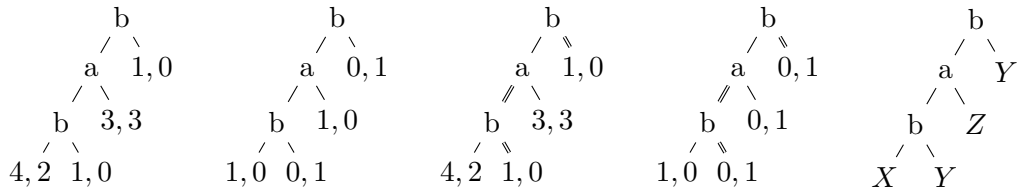
\begin{tabular}{ccccc}
\begin{tikzpicture}[level distance=7mm]
\node{b}[sibling distance=8mm]
	child{node{a}[sibling distance=8mm]
		child{node{b}[sibling distance=8mm]
			child{node{$4,2$}}
			child{node{$1,0$}}
		}
		child{node{$3,3$}}
	}
	child{node{$1,0$}};
\end{tikzpicture}
&
\begin{tikzpicture}[level distance=7mm]
\node{b}[sibling distance=8mm]
	child{node{a}[sibling distance=8mm]
		child{node{b}[sibling distance=8mm]
			child{node{$1,0$}}
			child{node{$0,1$}}
		}
		child{node{$1,0$}}
	}
	child{node{$0,1$}};
\end{tikzpicture}
&
\begin{tikzpicture}[level distance=7mm]
\node{b}[sibling distance=8mm]
	child{node{a}[sibling distance=8mm]
		child{node{b}[sibling distance=8mm] edge from parent[double]
			child{node{$4,2$}}
			child{node{$1,0$} edge from parent[double]}
		}
		child{node{$3,3$}}
	}
	child{node{$1,0$}edge from parent[double]};
\end{tikzpicture}
&
\begin{tikzpicture}[level distance=7mm]
\node{b}[sibling distance=8mm]
	child{node{a}[sibling distance=8mm]
		child{node{b}[sibling distance=8mm] edge from parent[double]
			child{node{$1,0$}}
			child{node{$0,1$} edge from parent[double]}
		}
		child{node{$0,1$}}
	}
	child{node{$0,1$} edge from parent[double]};
\end{tikzpicture}
&
\begin{tikzpicture}[level distance=7mm]
\node{b}[sibling distance=8mm]
	child{node{a}[sibling distance=8mm]
		child{node{b}[sibling distance=8mm]
			child{node{$X$}}
			child{node{$Y$}}
		}
		child{node{$Z$}}
	}
	child{node{$Y$}};
\end{tikzpicture}
\end{tabular}

A strategy profile is an ordered pair enclosing one strategy per player, where a strategy of a player is a complete collection of the (unique) choices that the player would make at each node (MODIFY that he/she owns) if the play ever reached this node. The third (resp. fourth) object from the left above represents a strategy profile for the first (resp. second) game, where the double lines represent the strategical choices. A strategy profile induces one unique payoff function, by following the unique choices from the root to the leaves. When assumed that the players prefer greater payoffs for themselves, the third object above is not a Nash equilibrium, because at least one of the player, namely $b$, can improve upon his/her payoff by changing his/her strategy from "right-right" to "left-left" and obtain $2$ instead of $0$. The fourth object above is a Nash equilibrium because no player can improve upon what they already have. Said otherwise, "right-right" is a winning strategy for player $b$.

The rightmost object above, with variables $X$, $Y$ and $Z$, is an abstraction of the leftmost game: it represents the structure of the game. Especially, repetition of the variable $Y$ captures equality between payoff functions in the original game. There are $2^3$ possibilities to instantiate the variables of the game structure with $(1,0)$ and $(0,1)$, one of them being the second game above. There are infinitely many possibilities to instantiate the variables of the game structure with real-valued payoff functions, one of them being the leftmost game above. The two strategy profiles above suggest that there are many more Nash equilibria in win-lose games than in games with real-valued payoff functions. This is actually true, but since all the $2^3$ win-lose games that are derived from the game structure above have a Nash equilibrium (equivalently, a winning strategy), which is easy to check, the equilibrium-transfer theorem, stated in Section~\ref{sect:mr-psc}, implies that all the games with real-valued payoff functions that are derived from the same game structure also have a Nash equilibrium, which is more difficult to check! (Albeit already proved in~\cite{Kuhn53}.)

In the same way, the equilibrium-transfer theorem turns essentially all determinacy results into existence of Nash equilibrium in a two-player, multi-outcome setting. It is applied to Borel determinacy, finite-memory determinacy of Muller game, and positional determinacy of parity games in Section~\ref{sect:app-et}. Apart from the generalisation of positional determinacy of parity games, which is new, the two other obtained results are weaker than existing results; but the key point here, however, is that the three applications are almost effortless, and that the same would hold for further applications.

The equilibrium-transfer theorem relates to all two-player win-lose games where the notion of winning strategy makes sense, namely strategy-based games. Since on the one hand, all strategy-based games are faithfully embeddable into games in normal form, as far as existence of Nash equilibrium is concerned, and since on the other hand, games with abstract outcomes and preferences are much more general than real-valued games, the theorem will be stated for abstract games in normal form.

\subsection{Definitions}\label{sect:def}

Unlike traditional games in normal form, the definition below involves abstract outcomes (instead of mere real-valued payoff functions) and preferences that are arbitrary (instead of transitive, reflexive, total, \textit{etc.}). It is important since there is no reason why games, \textit{e.g.}, with real-valued payoff functions should account for all possible games. For instance, Section~\ref{sect:c} defines a simple preference that has a game-theoretic property relating to Nash equilibrium but whose every linear extension fails to have the same game-theoretic property. 

\begin{definition}[Games in normal form]\label{defn:gnf}
They are tuples $\langle A,(S_a)_{a\in A},O,v,(\prec_a)_{a\in A}\rangle$ satisfying the following: 
\begin{itemize}
\item $A$ is a non-empty set (of players, or agents),
\item $\prod_{a\in A}S_a$ is a non-empty Cartesian product (whose elements are the strategy profiles and where $S_a$ represents the strategies available to player $a$),
\item $O$ is a non-empty set (of possible outcomes),
\item $v:\prod_{a\in A} S_a\to O$ (the outcome function that values the strategy profiles),
\item Each $\prec_a$ is a binary relation over $O$ (modelling the preference of player $a$).
\end{itemize}
\end{definition}

The traditional notion of Nash equilibrium is rephrased below in the abstract setting with a subtle semantic change (but remains the same in extension): each binary relation $\prec_a$, which I call preference, is the complement of the inverse of what is traditionally called preference. 

\begin{definition}[Nash equilibrium]\label{defn:ne}
Let $g=\langle A,(S_a)_{a\in A} ,O,v,(\prec_a)_{a\in A}\rangle$ be a game in normal form. A strategy profile $s$ in $S:=\prod_{a\in A} S_a$ is a Nash equilibrium if it makes every player $a$ stable, \textit{i.e.} $v(s)\not\prec_a v(s')$ for all $s'\in S$ that differ from $s$ at most at the $a$-component.
\[NE(s)\quad:=\quad\forall a\in A,\forall s'\in S,\quad\neg(v(s)\prec_a v(s')\,\wedge\,\forall b\in A-\{a\},\,s_b= s'_b)\]
\end{definition}

Three games in normal form are represented below as arrays. They all involve two players, say $a$ and $b$, two strategies for $a$ (resp. $b$), namely $a_l$ and $a_r$ (resp. $b_l$ and $b_r$), and outcomes in $\mathbb{R}^2$. In this specific case, an outcome is usually called a payoff function because it assigns one payoff to every player. Player $a$ (resp. $b$) prefers payoff functions with greater first (resp. second) component. In the first game, if player $a$ picks the strategy $a_l$ and player $b$ picks $b_l$, the strategy profile $(a_l,b_l)$ then yields payoff $1$ for $a$ and $0$ for $b$. This profile is not a Nash equilibrium because, by changing strategies, player $a$ can convert the profile $(a_l,b_l)$ into $(a_r,b_l)$ and obtain payoff $2$, which is greater than $1$. The game has two Nash equilibria, namely profiles $(a_r,b_l)$ and $(a_l,b_r)$. The second game has no Nash equilibrium and the third game, which  enjoys some symmetry, has two Nash equilibria. These last two games suggest that the notion of Nash equilibrium cannot, at least directly, and even by using probabilities, lead to a notion of best move, \textit{i.e.} of recommendations on how to play.

\begin{displaymath}
\begin{array}{c@{\hspace{1cm}}c@{\hspace{1cm}}c}
\begin{array}{c|c@{,\;}c@{\;\vline\;}c@{,\;}c|}
          \multicolumn{1}{c}{}&
	  \multicolumn{2}{c}{b_l}&
	  \multicolumn{2}{c}{b_r}\\
	  \cline{2-5}
 	  a_l & 1 & 0 & 5 & 0 \\
	  \cline{2-5}
	  a_r & 2 & 4 & 5 & 3\\
	  \cline{2-5}
\end{array}
&
\begin{array}{c|c@{,\;}c@{\;\vline\;}c@{,\;}c|}
          \multicolumn{1}{c}{}&
	  \multicolumn{2}{c}{b_l}&
	  \multicolumn{2}{c}{b_r}\\
	  \cline{2-5}
 	  a_l & 0 & 1 & 1 & 0 \\
	  \cline{2-5}
	  a_r & 1 & 0 & 0 & 1\\
	  \cline{2-5}
	 \end{array}
&
\begin{array}{c|c@{,\;}c@{\;\vline\;}c@{,\;}c|}
          \multicolumn{1}{c}{}&
	  \multicolumn{2}{c}{b_l}&
	  \multicolumn{2}{c}{b_r}\\
	  \cline{2-5}
 	  a_l & 2 & 1 & 0 & 0 \\
	  \cline{2-5}
	  a_r & 0 & 0 & 1 & 2\\
	  \cline{2-5}
	 \end{array}
\end{array}
\end{displaymath}

The two-player win-lose games in normal form, defined below, are special cases of games in normal form.

\begin{definition}[Win-lose games in normal form, winning strategies, and determinacy]\label{def:win-lose}

\begin{itemize}
\item A win-lose game is a game where $A=\{1,2\}$ and $O=\{(1,0),(0,1)\}$ and the preferences are defined by $(0,1)\prec_1(1,0)$ and $(1,0)\prec_2(0,1)$, so all these usually may remain implicit.
\item A winning strategy for player $1$ is a strategy $s_1\in S_1$ such that $v(s_1,s_2)=(1,0)$ for all $s_2\in S_2$. A winning strategy for player $2$ is a strategy $s_2\in S_2$ such that $v(s_1,s_2)=(0,1)$ for all $s_1\in S_1$.
\item A win-lose game such that one player has a winning strategy is said to be determined.
\end{itemize}
\end{definition}

Three win-lose games are represented below. For the first game, player $a$ has the winning strategy $a_r$; for the second game, player $b$ has the winning strategy $b_r$; there is no winning strategy for the third game.

\begin{displaymath}
\begin{array}{c@{\hspace{1cm}}c@{\hspace{1cm}}c}
\begin{array}{c|c@{,\;}c@{\;\vline\;}c@{,\;}c|}
          \multicolumn{1}{c}{}&
	  \multicolumn{2}{c}{b_l}&
	  \multicolumn{2}{c}{b_r}\\
	  \cline{2-5}
 	  a_l & 0 & 1 & 0 & 1 \\
	  \cline{2-5}
	  a_r & 1 & 0 & 1 & 0\\
	  \cline{2-5}
\end{array}
&
\begin{array}{c|c@{,\;}c@{\;\vline\;}c@{,\;}c|}
          \multicolumn{1}{c}{}&
	  \multicolumn{2}{c}{b_l}&
	  \multicolumn{2}{c}{b_r}\\
	  \cline{2-5}
 	  a_l & 1 & 0 & 0 & 1 \\
	  \cline{2-5}
	  a_r & 1 & 0 & 0 & 1\\
	  \cline{2-5}
	 \end{array}
&
\begin{array}{c|c@{,\;}c@{\;\vline\;}c@{,\;}c|}
          \multicolumn{1}{c}{}&
	  \multicolumn{2}{c}{b_l}&
	  \multicolumn{2}{c}{b_r}\\
	  \cline{2-5}
 	  a_l & 0 & 1 & 1 & 0 \\
	  \cline{2-5}
	  a_r & 1 & 0 & 0 & 1\\
	  \cline{2-5}
	 \end{array}
\end{array}
\end{displaymath}

The notion of winning strategy is relevant in win-lose games only, but the following remark clarifies why the transfer from winning strategy to multi-outcome Nash equilibrium is a process of generalisation.

\begin{remark}\label{rmk:ws-ne} 
A win-lose game has a winning strategies iff it has a Nash equilibrium.
\end{remark}

\begin{proof}
On the one hand, the strategy profile made of a winning strategy of a player and any strategy of his or her opponent constitutes a Nash equilibrium; conversely, the $X$-component of a Nash equilibrium where player $X$ wins is a winning strategy for $X$.
\end{proof}

Section~\ref{sect:scgc} has already exemplified that both win-lose games and abstract games can be derived from a game structure. This is formalised below.

\begin{definition}[Induced structures, derived games, determined structures, enforcement]\label{defn:is-dg-ds}\hfill

Let $\langle \{1,2\},S_1,S_2,O,v,\{\prec_1,\prec_2\}\rangle$ be a two-player game.
\begin{itemize}
\item $\langle \{1,2\},S_1,S_2,O,v\rangle$ is the structure induced by the game, and conversely, the game is said to be derived from the structure. 
\item Let $wl$ be a function from $O$ to $\{(1,0),(0,1)\}$, the win-lose game $\langle S_1,S_2,wl\circ v\rangle$ is also said to be derived from the structure. 
\item Let $R_1\subseteq S_1$ and $R_2\subseteq S_2$. If all win-lose games derived from a structure are determined (\textit{via} strategies in $R_1$ or $R_2$), the structure is also said to be determined (\textit{via} strategies in $R_1$ or $R_2$).
\item Let $\langle \{1,2\},S_1,S_2,O,v\rangle$ be a game structure, let $P\subseteq O$, and let $s_1\in S_1$ such that $v(s_1,S_2):=\{v(s_1,s_2)\,\mid\,s_2\in S_2\}\subseteq P$. The strategy $s_1$ is said to enforce $P$ and exclude $O\backslash P$.
\end{itemize}
\end{definition}

The subsets $R_i$ from Definition~\ref{defn:is-dg-ds} represent strategies of special interest. For instance, as already mentioned, Muller games are determined through strategies that can be described by finite automata, and parity games are determined through strategies that are called positional. 

The leftmost game structure below is not determined, \textit{e.g.}, because instantiating $X$ with $(1,0)$ and $Y$ with $(0,1)$ yields a game without Nash equilibrium, or equivalently without winning strategy. The other two game structures are determined. (To see this for the rightmost one, it suffices to make a case distinction on how $Y$ is instantiated.)

\begin{displaymath}
\begin{array}{c@{\hspace{1cm}}c@{\hspace{1cm}}c}
\begin{array}{c|c@{\;\vline\;}c|}
          \multicolumn{1}{c}{}&
	  \multicolumn{1}{c}{b_{l}}&
	  \multicolumn{1}{c}{b_{r}}\\
	  \cline{2-3}
 	  a_{l} & X & Y \\
	  \cline{2-3}
	  a_{r} & Y & X\\
	  \cline{2-3}
\end{array}
&
\begin{array}{c|c@{\;\vline\;}c|}
	  \multicolumn{1}{c}{}&
	  \multicolumn{1}{c}{b_{l}}&
	  \multicolumn{1}{c}{b_{r}}\\
       	  \cline{2-3}
 	  a_l & X & Z\\
	  \cline{2-3}
	  a_r & Y & Y\\
	  \cline{2-3}
\end{array}
&
\begin{array}{c|c@{\;\vline\;}c@{\;\vline\;}c|}
	  \multicolumn{1}{c}{}&
	  \multicolumn{1}{c}{b_{l}}&
	  \multicolumn{1}{c}{b_{m}}&
	  \multicolumn{1}{c}{b_{r}}\\
       	  \cline{2-4}
 	  a_l & X & Z &Y \\
	  \cline{2-4}
	  a_r & Y & Y &Y\\
	  \cline{2-4}
\end{array}
\end{array}
\end{displaymath}

The following remark holds since deriving a win-lose game from a structure amounts to choosing the characteristic function of a subset of the outcomes.

\begin{remark} \label{rmk:d-e}
A game structure is determined iff each subset of the outcomes can be either enforced by player $1$ or excluded by player $2$.  
\end{remark}

\subsection{Main result and proof in the simplest case}\label{sect:mr-psc}

Put simply, the equilibrium-transfer theorem reads as follows: if a game structure is determined through nice strategies, all (reasonable) abstract games derived from the structure have a nice Nash equilibrium; and the converse is of course true! (In the sequel, if $\prec$ is a binary relation, its inverse is defined by $x \prec^{-1} y$ iff $y \prec x$.)

\begin{theorem}[Equilibrium transfer]\label{thm:intro-et}
Let $\langle \{1,2\},S_1,S_2,O,v,\{\prec_1,\prec_2\}\rangle$ be a two-player game whose induced structure is determined through strategies in $R_1\subseteq S_1$ and $R_2\subseteq S_2$, and assume either of the following conditions:
\begin{enumerate}
\item the preferences $\prec_1$ (resp. $\prec_2$) has (uniformly) finite height,\\
\textit{i.e.} there is $n\in\mathbb{N}$ such that there is no $o_0\prec_1 o_1\prec_1 \dots\prec_1 o_n$ (resp. $o_0\prec_2 o_1\prec_2 \dots\prec_2 o_n$).
\item the strategy sets $S_1$ and $S_2$ are countable and $\prec_1^{-1}$ and $\prec_2^{-1}$ are well-founded,\\
\textit{i.e.} there is no infinite ascending sequence $o_0\prec_1 o_1\prec_1 \dots$ (resp. $o_0\prec_2 o_1\prec_2 \dots$)
\end{enumerate}
Then $\langle \{1,2\},S_1,S_2,O,v,\{\prec_1,\prec_2\}\rangle$ has a Nash equilibrium in $R_1\times R_2$.
\end{theorem}

Following up on Section~\ref{sect:scgc}: to prove that all finite real-valued two-player games in extensive form have a Nash equilibrium, it suffices, firstly, to show that the structures induced by these games are determined, secondly, to argue that the preferences of each given real-valued game have finite height (since they are acyclic and the game involves only finitely many payoff functions), and finally, to invoke the equilibrium-transfer theorem. Note that theses games are not derived from the same game structure, but rather from infinitely many structures.

The equilibrium-transfer theorem should be of interest to both game theorists and mathematicians dealing with games. To game theorists, it provides an economical way of proving existence of multi-outcome Nash equilibrium for all games in a class that is derived from a set of two-player game structures; to mathematicians, it provides an easy way to generalise their determinacy results, which are dedicated to logic in the first place, and thus advertise their work to game theorists.

Proposition~\ref{prop:mmt} below is a very simple version of the equilibrium-transfer theorem, but the very basic idea is already there. It uses only an abstraction of zero-sum games, \textit{i.e.} two-player games where the two preference relations are inverses of each other. It is named after von Neumann's Minimax Theorem (see, \textit{e.g.}, \cite{NM44}) to hint at the similarities, although it is neither a generalisation nor a special case of it. (For instance, the Minimax Theorem involves infinitely many mixed strategies whereas Proposition~\ref{prop:mmt} involves finitely many pure strategies only.) 

\begin{proposition}[Minimax transfer]\label{prop:mmt}
Let $\langle \{1,2\}, S_1,S_2,O,v,\{<_1,<_1^{-1}\}\rangle$ be a two-player game in normal form whose induced structure is determined, and assume that $<_1$ is a strict linear order over the finite domain $O$. Then the game has a Nash equilibrium (and all Nash equilibria yield the same outcome).
\end{proposition}

\begin{proof}
Let $P$ be the smallest (for inclusion) $<_1$-terminal interval (\textit{i.e.} $x\in P\wedge x<_1 y\Rightarrow y\in P$) that player $1$ can enforce \textit{via} some strategy $s_1$, and let $m$ be the $<_1$-minimum (and the $<_1^{-1}$-maximum!) of $P$. Since player $1$ cannot enforce $P-\{m\}$, by remark~\ref{rmk:d-e} player $2$ can enforce $(O\backslash P)\cup\{m\}$ \textit{via} some strategy $s_2$. Therefore $(s_1,s_2)$ is a Nash equilibrium yielding outcome $m$, and any strategy profile that does not yield $m$ may be improved upon by one player $i$ \textit{via} $s_i$.  
\end{proof}

All the conditions of application of Theorem~\ref{thm:intro-et} are important, as mentioned below and further detailed in Section~\ref{sect:ltt}. First, when the inverse of some preference is not well-founded, it is easy to build a (one-player) game without Nash equilibrium. Second, this article defines a two-player game without Nash equilibrium although its induced game structure is determined, while the strategies of one player only are countably many, one preference only has finite height, and the inverse of the other preference is still well-founded. Third, a natural three-player version of the equilibrium-transfer theorem may sound as follows: "given a three-player game with preferences of finite height, if replacing the actual preferences by preferences of smaller height always yields a game with a Nash equilibrium, and/or if merging two players (into a super-player) or slicing the induced structure (\textit{i.e} fixing one player's strategy) always yields a determined structure, then the given three-player game has a Nash equilibrium." Counter-examples show that simpler versions of the statement above do not hold, but the general case is still open.

\section{The equilibrium-transfer theorem}\label{sect:tt}

This section proves the theorem by transfinite induction on the preferences. The three main ingredients of the proof are: an equilibrium-reflecting reduction that shrinks games in terms of preferences, a property on functions from $\mathbb{N}^2$ to $\mathbb{N}$ that enables a diagonal argument when shrinking games is not possible, and a finite-case version of the theorem, which itself relies on lifting binary relations to the power set of their domains. This lift, defined below, is the key idea of the equilibrium transfer: especially, it overcomes the difficulty that the proof of the minimax transfer does not scale up for preferences that are not inverses of each other. Note that a simpler and sufficient (to prove equilibrium transfer) version of the lift is mentioned in Remark~\ref{rem:finite-linear} afterwards.

\begin{definition}\label{defn:lift}
A binary relation $\prec$ on a set $S$ may be lifted to the power set of $S$ as below.
\[\forall A,B\subseteq S,\quad A\prec^{\mathcal{P}} B\,:=\,\exists a\in A\backslash B,\forall b\in B\backslash A,\,a\prec b\] 
\end{definition}

\begin{lemma}\label{lem:lrs-lp}
Let $\prec$ be a binary relation on a set $S$. If $\prec$ is a strict linear order, $\prec^{\mathcal{P}}$ is a strict partial order. 
\end{lemma}

\begin{proof}
A strict partial order is a transitive and irreflexive binary relation. A strict linear order is a strict partial order such that any two distinct elements are comparable. Assume that $\prec$ is as strict linear order. Since $\prec^{\mathcal{P}}$ is irreflexive by definition, it suffices to show that $\prec^{\mathcal{P}}$ is transitive. Assume that $A\prec^{\mathcal{P}}B$ and $B\prec^{\mathcal{P}}C$ with respective witnesses $a\in A\backslash B$ and $b\in B\backslash C$. First note that $a\neq b$ since $a\notin B$ and $b\in B$. Now let us case-split to show that $A\prec^{\mathcal{P}}C$.
\begin{itemize}
\item Assume that $a\prec b$, so $\neg(b\prec a)$ by transitivity and irreflexivity assumptions, so $a\notin C\backslash B$ since $b$ is a witness for $B\prec^{\mathcal{P}}C$. Together with $a\notin B$ it yields $a\notin C$, so $a\in A\backslash C$. Now let $x$ be in $C\backslash A$. If $x\in B$, then $x\in B\backslash A$, and $a\prec x$ since $a$ is a witness for $A\prec^{\mathcal{P}}B$. If $x\notin B$, then $x\in C\backslash B$, and $b\prec x$ since $b$ is a witness for $B\prec^{\mathcal{P}}C$, so $a\prec x$ by transitivity. Therefore $A\prec^{\mathcal{P}}C$ is witnessed by $a$.
\item Assume that $b\prec a$, so $\neg(a\prec b)$ by transitivity and irreflexivity assumptions, so $b\notin B\backslash A$ since $a$ is a witness for $A\prec^{\mathcal{P}}B$. Together with $b\in B$ it yields $b\in A$, so $b\in A\backslash C$. Now let $x$ be in $C\backslash A$. If $x\notin B$, then $x\in C\backslash B$, and $b\prec x$ since $b$ is a witness for $B\prec^{\mathcal{P}}C$. If $x\in B$, then $x\in B\backslash A$, and $a\prec x$ since $a$ is a witness for $A\prec^{\mathcal{P}}B$, so $b\prec x$ by transitivity. Therefore $A\prec^{\mathcal{P}}C$ is witnessed by $b$.
\end{itemize}
\end{proof}

\begin{remark}\label{rem:finite-linear}
For finite linear order, Definition~\ref{defn:lift} can be rephrased as $A<^{\mathcal{P}} B\,:= A\neq B\,\wedge\, \mathrm{min}_{<}(A\backslash B \cup B\backslash A)\in A$, and Lemma~\ref{lem:lrs-lp} can be strengthen into "$<^{\mathcal{P}}$ is also a strict linear order". This $<^{\mathcal{P}}$ is isomorphic to the lexicographic order induced by $<$ on the characteristic functions of the complements of the subsets. For example let $o_1 < o_2 < o_3 < o_4 < o_5$, then 
$\{o_2,o_3,o_4,o_5\} <^{\mathcal{P}} \{o_2,o_4\}$ corresponds to $10000 <_{lex} 10101$. 
\end{remark}

The lemma below states a bit more than a mere finitely-many-outcome version of the forthcoming theorem; it sounds a bit less natural too, due to Condition~\ref{cond:lem-et3} (which is obviously fulfilled when the outcomes are finitely many), but it is very useful in the proof of the equilibrium-transfer theorem. In the statement of the lemma, $\prec_i^*$ denotes the reflexive and transitive closure of $\prec_i$.

\begin{lemma}[Finitary equilibrium transfer]\label{lem:et}
Let $\langle \{1,2\}, S_1,S_2,O,v,\{\prec_1,\prec_2\}\rangle$ be a two-player game in normal form, let $R_1\subseteq S_1$ and $R_2\subseteq S_2$, and let us assume the following:
\begin{enumerate}
\item\label{cond:lem-et1} the game structure is determined \textit{via} strategies in $R_1$ and $R_2$.
\item\label{cond:lem-et2} both preferences $\prec_1$ and $\prec_2$ are acyclic.
\item\label{cond:lem-et3} $\exists s_1\in S_1,\,|\{o\in O\,\mid\,\exists s_2\in S_2,\,v(s_1,s_2)\,\prec_1^*\, o\}|<\infty\quad\vee$\\
$\exists s_2\in S_2,\,|\{o\in O\,\mid\,\exists s_1\in S_1,\,v(s_1,s_2)\,\prec_2^*\, o\}|<\infty$\\
That is, one player $i$ can enforce a subset of outcomes whose $\prec_i$-upward-generated cone is finite.
\end{enumerate}
Then the game $\langle \{1,2\}, S_1,S_2,O,v,\{\prec_1,\prec_2\}\rangle$ has a Nash equilibrium in $R_1\times R_2$.
\end{lemma}

\begin{proof}
First note that if player $i$ can enforce a subset of outcome (\textit{via} some strategy in $S_i$), he or she can enforce it \textit{via} some strategy in $R_i$, since the opponent cannot exclude it and by determinacy assumption together with Remark~\ref{rmk:d-e}. Assume that, \textit{e.g.}, player $1$ can enforce a finite $\prec_1$-upward-generated cone $C$. Since $\prec_1$ is acyclic, so is its restriction $\prec_1\mid_C$ to $C$; let $<$ be a strict linear extension of $\prec_1\mid_C$, so $<^{\mathcal{P}}$ is a strict linear order too, by Remark~\ref{rem:finite-linear}. Let $M$ be the $<^{\mathcal{P}}$-greatest subset of $C$ that player $1$ can enforce and let $s_1\in R_1$ be a strategy enforcing $M$. 

Since $M$ is finite and non-empty and since $\prec_2$ is acyclic, let $m$ be $\prec_2\mid_M$-maximal and let $X:=\{x\in M\,\mid\,  x<m\}\cup\{x\in C\,\mid\,m<x\}$. Since $M<^{\mathcal{P}}X$ by Remark~\ref{rem:finite-linear} (or by Definition~\ref{defn:lift}), and since $X\subseteq C$ by definition of $M$ and $X$, player $1$ cannot enforce $X$ by definition of $M$, so player $2$ can enforce $O\backslash X$ by determinacy assumption and Remark~\ref{rmk:d-e}. Let $s_2\in R_2$ be a strategy enforcing $O\backslash X$, so that $v(s_1,s_2)\in M\cap (O\backslash X)=\{m\}$. 

First, the strategy profile $(s_1,s_2)$ makes Player $2$ stable, since $m$ is $\prec_2$-maximal among $M$, which is enforced by $s_1$. Second, let $o\in O$ be such that $m\prec_1 o$, so $o\in C$ by definition of $C$, so $o\in X$ by definition of $<$ and $X$, so $m$ is $\prec_1$-maximal among $O\backslash X$, which is enforced by $s_2$. Therefore $(s_1,s_2)\in R_1\times R_2$ is a Nash equilibrium.
\end{proof}

There is a straightforward algorithmic consequence of the proof of Lemma~\ref{lem:et}. Namely, finding a suitable Nash equilibrium in a two-player game that involves $n$ outcomes requires at most $n$ (resp. 2) calls to the function expecting a win-lose game and returning the winning player (resp. a suitable winning strategy). To justify this, let us fix a game $\langle \{1,2\}, S_1,S_2,\{o_1,\dots,o_n\},v,\{\prec_1,\prec_2\}\rangle$ and let $<$ be a linear extension of $\prec_1$ and assume that $o_1 < \dots < o_n$ up to renaming. Let us represent every subset $O'$ of $\{o_1,\dots,o_n\}$ via a characteristic word $u$ over $\{0,1\}$ (which may be seen as characteristic functions), where $u_i=1$ iff $o_i\in O'$. For all $u \in \{0,1\}^n$ let $w_a(u)$ be the winner of the derived win-lose game $\langle S_1,S_2, u\circ v\rangle$ (see Definition~\ref{def:win-lose}) and $w_s(u) \in R_1 \sqcup R_2$ be a winning strategy for the winner. When called with the arguments $(n,\epsilon)$, the function $g$ defined below returns within $n$ recursive calls the characteristic word of the $<$-maximum subset that player $1$ can enforce, where each recursive step calls $w_a$ only once.

\begin{itemize}
\item $g(0,u) := u$
\item $g(k+1,u):= g(k, u\cdot b)$ where $b := 0$ if $w_a(u\cdot 0\cdot 1^{k}) = 1$ and $b := 1$ otherwise.
\end{itemize}

Now let $o_j$ be $\prec_2$-maximal in the set represented by $g(n,\epsilon)$. The strategy profile $(w_s\circ g(n,\epsilon),w_s(g(n,\epsilon)_{<j}\cdot 0\cdot 1^{n-j})$ is the Nash equilibrium from the proof of Lemma~\ref{lem:et}.

When considering infinitely many outcomes, there may not exist a maximal subset that a given player can enforce. Nonetheless, the lemma above and the following two lemmas will be combined to prove the theorem by transfinite induction on (the order types of the inverses of) the preferences. 

Lemma~\ref{lem:reflect-eq} below relies on the remark that if a player can exclude a lower/downward interval of least-preferred outcomes, no Nash equilibrium will ever yield such outcomes. So, the excludable least-preferred outcomes may just be merged into one single worst outcome of the player, and become the best outcome of the opponent: indeed, this reduction does not create any Nash equilibrium but yields in many cases a smaller game in terms of outcomes and especially of preferences, thus enabling a step in the transfinite induction. Lemma~\ref{lem:reflect-eq} is named after the well-known elimination of dominated strategies (see, \textit{e.g.}, \cite{LR57}), which simplifies a game through its set of strategies only. The two procedures have nothing much in common, but the naming is meant to suggest that they may complement each other nicely (although not in this article).

\begin{lemma}[Elimination of dominated outcomes]\label{lem:reflect-eq}
Let $g=\langle \{1,2\}, S_1,S_2,O,v,\{<_1,<_2\}\rangle$ be a two-player game in normal form with strict linear preferences. Let $e\in S_1$ and $o\in O$ and assume that $o<_1v(e,s_2)$ for all $s_2\in S_2$. Let $g^\prime:=\langle \{1,2\}, S_1,S_2,O^\prime,v^\prime,\{<^\prime_1,<^\prime_2\}\rangle$, where
\begin{itemize}
\item $O^\prime :=\{x\in O\,\mid\, o\leq_1 x\}$
\item $v^\prime(s):=v(s)$ if $v(s)\in O^\prime$ and  $v^\prime(s):=o$ otherwise.
\item $<^\prime_1$ is the restrictions of $<_1$ to $O^\prime$.
\item $x<^\prime_2 y:= x\neq o\,\wedge\, (x<_2 y\,\vee\, y=o)$. 
\end{itemize}

Then every Nash equilibrium of $g^\prime$ is also Nash equilibrium of $g$. 

Moreover, if the inverse relations of $<_1$ and $<_2$ are well-orders, the order types of $(<'_1)^{-1}$ and $(<'_2)^{-1}$ are not greater than those of $(<_1)^{-1}$ and $(<_2)^{-1}$ respectively. Furthermore, if $o$ is not the $<_1$-least element of $O$, the order type of $(<'_1)^{-1}$ is less than that of $(<_1)^{-1}$.
\end{lemma}

\begin{proof}
Let $s$ be a Nash equilibrium of $g^\prime$. Since $o<_1 v(e,s_2)$ by assumption about $e$, the outcome $v(e,s_2)$ is in $O^\prime-\{o\}$ by definition of $O^\prime$, so $o<^\prime_1v^\prime(e,s_2)$ by definitions of $v^\prime$ and $<^\prime_1$. Since $v^\prime(e,s_2)\leq^\prime_1 v^\prime(s)$ by definition of NE and since $<^\prime_1$ is also strict linear, we have $o<^\prime_1v^\prime(s)$, so $v^\prime(s)\in O^\prime-\{o\}$ and $v(s)=v^\prime(s)$ by definitions of $O^\prime$ and $v^\prime$.

Now let us prove by contradiction that both players are stable w.r.t. $s$ and $g$. If $v(s)<_1v(x,s_2)$ for some $x\in S_1$, then $v(x,s_2)\in O^\prime$ since $v(s)\in O^\prime$ and by definition of $O^\prime$, so $v^\prime(s)<^\prime_1v^\prime(x,s_2)$ by definitions of $v^\prime$ and $<^\prime_1$, which contradicts $s$ being an NE of $g^\prime$. If $v(s)<_2v(s_1,y)$ for some $y\in S_2$, then $v^\prime(s)=o$ since $v^\prime(s_1,y)\leq^\prime_2 v^\prime(s)$ (by definition of NE) and by definition of $<^\prime_2$, which is also a contradiction.
\end{proof}

The transfinite-inductive step of the proof of Theorem~\ref{thm:et} will be justified by Lemma~\ref{lem:reflect-eq} above; but in cases where no player is able to exclude a non-trivial downward interval of outcomes, existence of sets $A$ and $B$ will be proved to feed Lemma~\ref{lem:countable-finite} below, which will subsequently yield the existence of a set $C$ that contradicts the the determinacy assumption.

\begin{lemma}\label{lem:countable-finite}
Let $f:\mathbb{N}^2\to\mathbb{N}$. The following two propositions are equivalent.
\begin{enumerate}
\item There exists a subset of the naturals that intersects each $f(n,\mathbb{N})$ and whose complement intersects each $f(\mathbb{N},n)$, where $f(n,\mathbb{N}):=\{f(n,k)\,\mid\,k\in\mathbb{N}\}$.
\item There exist $A$ and $B$ disjoint subsets of the naturals such that either $f(n,\mathbb{N})$ and $A$ overlap or $f(n,\mathbb{N})\backslash(A\cup B)$ is infinite, and likewise, either $f(\mathbb{N},n)$ and $B$ overlap or $f(\mathbb{N},n)\backslash(A\cup B)$ is infinite.
\end{enumerate}

The statement is formalised below.

\[\begin{array}{c}
\forall f:\mathbb{N}^2\to\mathbb{N},\\
\exists C\subseteq\mathbb{N},\forall n\in\mathbb{N},\quad f(n,\mathbb{N})\cap C\neq\emptyset\quad\wedge\quad f(\mathbb{N},n)\cap\mathbb{N}\backslash C\neq\emptyset\\
\Updownarrow\\
\exists A,B\subseteq\mathbb{N}, A\cap B=\emptyset\quad\wedge\quad\forall n\in\mathbb{N},\quad (f(n,\mathbb{N})\cap A\neq\emptyset\quad\vee\quad|f(n,\mathbb{N})\backslash(A\cup B)|=\aleph_0)\quad\wedge\\
\phantom{\exists A,B\subseteq\mathbb{N}, A\cap B=\emptyset\quad\wedge\quad\forall n\in\mathbb{N},\quad} (f(\mathbb{N},n)\cap B\neq\emptyset\quad\vee\quad|f(\mathbb{N},n)\backslash(A\cup B)|=\aleph_0)\phantom{\quad\wedge}
\end{array}\]
\end{lemma}

\begin{proof}
For the top-bottom implication $A:=C$ and $B:=\mathbb{N}\backslash C$ witness the claim. To prove the bottom-top implication let us define two sequences of subsets of $\mathbb{N}$ as follows, by mutual induction.
\begin{eqnarray*}
X_0  &:=& A\\
Y_0  &:=& B\\
X_{n+1} &:=& X_n\cup\{min (f(n,\mathbb{N})\backslash(X_n\cup Y_n))\}\mbox{ if } f(n,\mathbb{N})\cap A=\emptyset\mbox{, otherwise } X_{n+1} := X_n \\
Y_{n+1} &:=& Y_n\cup\{min (f(\mathbb{N},n)\backslash(X_{n+1}\cup Y_n))\}\mbox{ if } f(\mathbb{N},n)\cap B=\emptyset\mbox{, otherwise } Y_{n+1} := Y_n \\
\end{eqnarray*}
The inductive steps above are well-defined by the assumed disjunctions and since the $X_n\backslash A$ and $Y_n\backslash B$ are finite by construction. It is provable by induction on $n$ that $X_n$ and $Y_n$ are disjoint for all $n$, and so are $X:=\bigcup_{n=0}^{\infty}X_n$ and $Y:=\bigcup_{n=0}^{\infty}Y_n$. Now note that that $C:=X$ witnesses the claim since $X_{n+1}$ (resp. $Y_{n+1}$) intersects $f(n,\mathbb{N})$ (resp. $f(\mathbb{N},n)$) by construction. 
\end{proof}

The proof of Theorem~\ref{thm:et} starts with a case distinction on Condition~\ref{cond:thm-et2}. The first case is exactly Lemma~\ref{lem:et}; the second case is reduced to Lemma~\ref{lem:et}; and the third case is proved by transfinite induction on the preferences: the base step invokes Lemma~\ref{lem:et} again and the inductive step performs a case distinction, invoking Lemma~\ref{lem:reflect-eq} when possible and proving by Lemma~\ref{lem:countable-finite} that impossibility would contradict Condition~\ref{cond:thm-et1}.  (Note that Theorem~\ref{thm:et} proves slightly more than what was promised by Theorem~\ref{thm:intro-et}.)

\begin{theorem}[Equilibrium transfer]\label{thm:et}
Let  $\langle \{1,2\}, S_1,S_2,O,v,\{\prec_1,\prec_2\}\rangle$ be a two-player game in normal form, let $R_1\subseteq S_1$ and $R_2\subseteq S_2$, and let us assume the following:
\begin{enumerate}
\item\label{cond:thm-et1} the induced game structure is determined \textit{via} strategies in $R_1$ and $R_2$.
\item\label{cond:thm-et2} one of the following assertions holds:
\begin{itemize}
\item the preferences are acyclic and one player $i$ can enforce a finite $\prec_i$-upward cone.
\item the preferences have (uniformly) finite height.
\item $S_1$ and $S_2$ are countably many and the inverses of the preferences are well-founded.   
\end{itemize}
\end{enumerate}
Then the game $\langle \{1,2\}, S_1,S_2,O,v,\{\prec_1,\prec_2\}\rangle$ has a Nash equilibrium in $R_1\times R_2$.
\end{theorem}

\begin{proof}
The first case is proved by Lemma~\ref{lem:et}. For the second case, let $n\in\mathbb{N}$ bound the height of both $\prec_1$ and $\prec_2$, and let $\rho_1:O\to\{0,\dots,n-1\}$ and $\rho_2:O\to\{0,\dots,n-1\}$ be corresponding rank functions, that is, $x\prec_i y$ implies $\rho_i(x)<\rho_i(y)$. Consider the game $\langle \{1,2\}, S_1,S_2,\{0,\dots,n-1\}^2,(\rho_1\circ v,\rho_2\circ v),\{\prec'_1,\prec'_2\}\rangle$ where $(i,j)\prec'_1(k,l)$ iff $i<k$ and $(i,j)\prec'_2(k,l)$ iff $j<l$. By Lemma~\ref{lem:et} the derived game has a Nash equilibrium in $R_1\times R_2$, which happens to be a Nash equilibrium for the original game too, by property of $\rho_i$.

As for the third case, it suffices to prove the statement for well-ordered preferences, by linear extension and preservation of Nash equilibrium by (set-theoretic) inclusion. Without loss of generality, let us also assume that $S_1$ and $S_2$ are infinite, otherwise let us duplicate strategies, and that $O$ is countable, otherwise let us replace it with $v(S_1,S_2)$. Note that if player $i$ can enforce a subset of outcome (\textit{via} some strategy in $S_i$), he or she can enforce it \textit{via} some strategy in $R_i$, since the opponent cannot exclude it and by determinacy assumption together with Remark~\ref{rmk:d-e}. Let us define a well-founded binary relation over pairs of ordinals, as follows: $(\alpha,\beta)\prec(\gamma,\delta):=(\alpha<\gamma\,\wedge\,\beta\leq\delta)\,\vee\,(\alpha\leq\gamma\,\wedge\,\beta<\delta)$, where $<$ is the usual well-order over ordinals, and let us proceed with the proof by induction (w.r.t $\prec$) on the pairs of order types that correspond to the inverses of the preferences. 

If one order type is finite, it suffices to invoke Lemma~\ref{lem:et}, so let us deal with the case where both order types are infinite. If one player can exclude two or more of his/her least-preferred outcomes, Lemma~\ref{lem:reflect-eq} and the induction hypothesis prove the claim, so let us deal with the case where no player can exclude two or more of his/her least-preferred outcomes. Let $A$ be the set containing (when they exist) the $\prec_1$-least and second-$\prec_1$-least outcomes, so that $A$ is empty if the order-type for player $a$ is a limit ordinal and $A$ is a singleton when the order type is a limit ordinal plus one. Since player $a$ cannot exclude any non-trivial downward interval, every finite set of outcomes that he/she can enforce must intersect $A$. Let us define $B$ likewise and invoke the bottom-to-top implication of Lemma~\ref{lem:countable-finite} instantiated with, up to bijection, $A$ and $B$ defined above and $f:=v$. It implies the existence of a subset of the outcomes ($C$ in Lemma~\ref{lem:countable-finite}) that player $b$ cannot enforce and that player $a$ cannot exclude, which contradicts the determinacy assumption.
\end{proof}

\section{Applications of the equilibrium-transfer theorem}\label{sect:app-et}

Section~\ref{sect:gmt} generalises Martin's Theorem on Borel determinacy, from descriptive set theory, and Section~\ref{sect:pg-mg} generalises two determinacy results from theoretical informatics, namely positional determinacy of parity games and finite-memory determinacy of Muller games. Note that in each case the corresponding class of win-lose games is derived from infinitely many game structures (rather than from a single game structure), as similarly mentioned in Section~\ref{sect:mr-psc}.

\subsection{Generalisation of Borel Determinacy}\label{sect:gmt}

An infinite two-player alternate game consists of two players that play alternately and infinitely many times. In addition, the same non-empty set of choices $C$ is available at each stage, so the first player picks an element in $C$, then the second player picks an element in $C$, then the first player picks an element in $C$ again, and so on. The underlying structure of such a game is a leafless and uniform rooted tree. Moreover, each infinite sequence of choices is mapped to some outcome and both players have preferences over the outcomes.

\begin{definition}[Infinite two-player alternate games and strategies]
An infinite 2-player alternate game is an object $\langle C,O,v,\{\prec_1,\prec_2\}\rangle$ complying with the following:
\begin{itemize}
\item $C$ is a non-empty set (of choices).
\item $O$ is a non-empty set (of possible outcomes of the game). 
\item $v:C^\omega\to O$ (uses outcomes to value the infinite sequences of choice).
\item $\prec_1$ and $\prec_2$ are binary relations over $O$ (called the preferences of player $1$ and $2$, respectively).
\end{itemize}
A function of type $C^{2*}\to C$ (resp. $C^{2*+1}\to C$) is a strategy for player $1$ (resp. $2$). 
\end{definition}

A strategy of a player tells what he/she would play at each node of the game. Since the tree has a uniform structure, it is convenient to represent a strategy of the first (resp. second) player by a function of type $C^{2*}\to C$ (resp. $C^{2*+1}\to C$), where $C^{2*}$ represents the finite sequences on $C$ of even (resp. odd) length. When both players have chosen their individual strategies, their combination induces a unique play, \textit{i.e.} a unique infinite sequence of choices. 

\begin{definition}[Induced play]
Given a game $g=\langle C,W\rangle$, and $s_1:C^{2*}\to C$, and $s_2:C^{2*+1}\to C$, let us define $p(s_1,s_2)$ through its prefixes, inductively as below, where $p_{<n}$ is the prefix of $p$ of length $n$ and the symbol $\cdot$ represents concatenation.  
\begin{itemize}
\item $p(s_1,s_2)_{< 2n+1}:=p(s_1,s_2)_{< 2n}\cdot s_1(p(s_1,s_2)_{< 2n})$
\item $p(s_1,s_2)_{< 2n+2}:=p(s_1,s_2)_{< 2n+1}\cdot s_2(p(s_1,s_2)_{< 2n+1})$
\end{itemize}
\end{definition}

An infinite two-player alternate game $\langle C,O,v,\{\prec_1,\prec_2\}\rangle$ may be translated into a game in normal form $\langle\{1,2\}, (C^{2*}\to C)\times (C^{2*+1}\to C),O,v\circ p,\{\prec_1,\prec_2\}\rangle$, which provides the infinite two-player alternate game framework with a natural notion of Nash equilibrium.

Before stating Borel determinacy below, let us recall that a subset of a topological space $X$ is called Borel if it belongs to the smallest collection of subsets of $X$ which contains all the open sets and is closed under complementation and countable union.

\begin{theorem}[Martin~\cite{Martin75},~\cite{Martin85}]
Let $C$ be a non-empty set and $v:C^\omega\to\{(1,0),(0,1)\}$ be such that $v^{-1}\{(1,0)\}$ is a Borel set of $C^\omega$ (which is endowed with for the product topology of the discrete topology on $C$). Then the win-lose game $\langle (C^{2*}\to C)\times (C^{2*+1}\to C),v\circ p\rangle$ is determined. 
\end{theorem}

The generalisation of Martin's theorem below is a straightforward corollary of both Martin's Theorem itself and the equilibrium transfer theorem.  

\begin{corollary}\label{cor:gmt}
Let $\langle C,O,v,\{\prec_1,\prec_2\}\rangle$ be an infinite 2-player alternate game and assume the following three conditions.
\begin{itemize}
\item $O$ is countable.
\item $\prec_1$ and $\prec_2$ have (uniformly) finite height.
\item For all $o\in O$, the pre-image $v^{-1}\{o\}$ is Borel.
\end{itemize}
Then the game $\langle C,O,v,\{\prec_1,\prec_2\}\rangle$ has a Nash equilibrium.
\end{corollary}

\begin{proof}
Thanks to the above-mentioned embedding (of alternate games into games in normal form) and the (uniformly) finite height assumption, it suffices to check Condition~\ref{cond:thm-et1} from Theorem~\ref{thm:et}. This is done along Definition~\ref{defn:is-dg-ds}, so let $wl:O\to\{(1,0),(0,1)\}$. The set $(wl\circ v)^{-1}\{(1,0)\}$ is Borel since it equals $\cup\{v^{-1}\{o\}\,\mid\,v(o)=(1,0)\}$, a countable union of set that are Borel by assumtion. So, by Borel determinacy, the win-lose game $\langle (C^{2*}\to C)\times (C^{2*+1}\to C),w\circ v\circ p\rangle$ is determined.
\end{proof}

\subsection{Generalisations on parity games and Muller games}\label{sect:pg-mg}

These infinite two-player games are played on graphs. Unfolding these graphs yields infinite trees, so Borel determinacy may imply determinacy for these games. However, Borel determinacy does not say whether there exist simple winning strategies, or more generally winning strategies satisfying some predicate. So, the results that are generalised in Section~\ref{sect:pg-mg} are not mere corollaries of Borel determinacy. The definitions below rephrase, \textit{e.g.}, \cite{GW06}.

\begin{definition}[Arena and strategy]
An arena is an object $\langle V, V',E,C,\gamma\rangle$ complying with the following:
\begin{itemize}
\item $V$ is a non-empty set of vertexes.
\item $V'\subseteq V$ are the vertexes owned by player $1$.
\item $E\subseteq V\times V$ are the edges of a sink-free graph, \textit{i.e.} $\forall x\in V,xE:=\{y\in V\,\mid\,xEy\}\neq\emptyset$.
\item $C$ is a non-empty set of colours.
\item $\gamma:V\to C$ assigns a colour to each of the vertexes. 
\end{itemize} 
A strategy of player $1$ (resp. $2$) is a function of (dependent) type $V^*\to\forall v\in V',\, vE$ (resp. $V^*\to\forall v\in V\backslash V',\, vE$). A strategy profile is a function of (dependent) type $V^*\to\forall v\in V,\, vE$. The combination $(s_1,s_2)$ of a strategy $s_1$ for players $1$ and $s_2$ for player $2$ amounts to a strategy profile, which, when starting from a given vertex, induces a unique infinite sequence of colours $\Gamma(s_1,s_2)\in C^{\mathbb{N}}$. 
\end{definition}

\begin{definition}[multi-outcome priority/Muller games]
A multi-outcome priority (resp. Muller) game is an object $\langle\mathcal{G},O,r,\{\prec_1,\prec_2\}\rangle$ complying with the following:
\begin{itemize}
\item $\mathcal{G}$ is an arena where $C=\mathbb{N}$ (resp. $\mathcal{G}$ a is finite arena) as defined above.
\item $O$ is a non empty set of outcomes.
\item $r:\mathbb{N}\cup\{\bot\}\to O$ (resp. $r:\mathcal{P}(C)\to O$)
\item $\prec_1$ and $\prec_2$ are binary relations over $O$, the preferences.
\end{itemize}
For every infinite sequence of colours $\Gamma$, let $cl(\Gamma)$ be its cluster set, \textit{i.e.} the set of the colours occurring infinitely often in $\Gamma$. The outcome that is induced by a sequence of colours $\Gamma\in C^{\mathbb{N}}$, \textit{i.e.} by a strategy profile and a starting vertex, is: 
\begin{itemize}
\item For Muller games, $r\circ cl(\Gamma)$.
\item For priority games, $r\circ min\circ cl(\Gamma)$ if $cl(\Gamma)\neq\emptyset$, otherwise $r(\bot)$.
\end{itemize}
\end{definition}

Note that setting $O:=\{(1,0),(0,1)\}$ and $(0,1)\prec_1 (1,0)$ and $(1,0)\prec_2 (0,1)$ (plus $r(2n):=r(\bot):=(1,0)$ and $r(2n+1):=(0,1)$ for a multi-outcome priority game) in the definition above yields a parity (resp. Muller) game, up to isomorphism.

It was proved in~\cite{GH82} that Muller games are determined through finite-memory strategies and in~\cite{GW06} that parity games with priorities in $\mathbb{N}$ are positionally determined. Since these are determinacy results, let us extend them to multi-outcome settings below. (Note that one need not know what positional or finite-memory means.)

\begin{corollary}\label{cor:mo-mg}
Every multi-outcome Muller game (initiated with a starting vertex) with acyclic preferences has a finite-memory Nash equilibrium. 
\end{corollary}

\begin{proof}
Let $\langle V, V',E,C,\gamma,O,r,\{\prec_1,\prec_2\}\rangle$ be a multi-outcome Muller game with acyclic preferences, and let $v_0\in V$ be the starting vertex. Since it is naturally embedded into a game in normal form (as far as NE are concerned), it suffices to invoke Lemma~\ref{lem:et} to prove the claim, where Conditions~\ref{cond:lem-et2} and \ref{cond:lem-et3} are fulfilled by assumption and finiteness of the game, respectively. Finally, Condition~\ref{cond:lem-et1} is also fulfilled: indeed, for every $wl:O\to\{(1,0),(0,1)\}$, the derived win-lose game $\langle V, V',E,C,\gamma,wl\circ r\rangle$ is a Muller game, so by \cite{GH82} it is determined \textit{via} finite-memory strategies.
\end{proof}

\begin{corollary}\label{cor:mo-pg}
Every multi-outcome priority game where preferences have finite height has a positional Nash equilibrium.
\end{corollary}

\begin{proof}
The proof is similar to the proof of Corollary~\ref{cor:mo-mg}, up to one point: let $wl:O\to\{(1,0),(0,1)\}$, the derived win-lose game $\langle V, V',E,C,\gamma,wl\circ r\rangle$ is not obviously isomorphic to a parity game, because the function $wl\circ r$ may not map even numbers and $\bot$ to $(1,0)$, and odd numbers to $(0,1)$. This is easily overcome by renaming the colours: let $\gamma'(v):=2\cdot\gamma(v)$ if $wl\circ r\circ\gamma(v)=wl\circ r(\bot)$ and $\gamma'(v):=2\cdot\gamma(v)+1$ otherwise. By \cite{GW06} the parity game $\langle V, V',E,C,\gamma'\rangle$ is determined, and so is $\langle V, V',E,C,\gamma,wl\circ r\rangle$.
\end{proof}

Note that, in the area of graph games for program verification, \cite{Ummels05} has already investigated extensions of determinacy in various directions, namely for subgame perfect equilibrium (a stronger notion of Nash equilibrium), for $n$-player games instead of two-player games, or for payoff functions in $\{0,1\}^n$ instead of $\{(0,1),(1,0)\}$. For instance, Theorem 4.19. in \cite{Ummels05} states that any initialised two-player parity game has a positional subgame perfect equilibrium and Theorem 4.20. states that any initialised finite multiplayer parity game has a finite-state subgame perfect equilibrium.

\section{Limitations of transfer possibilities}\label{sect:ltt}

Section~\ref{sect:lim-order-set} below suggests that the order and set-theoretic assumptions of Theorem~\ref{thm:et} are tight, although the case where exactly one strategy set is countable and the preferences are well-founded yet with chains of arbitrary length is still open; Section~\ref{sect:3player} afterwards suggests that the two-player assumptions of Theorem~\ref{thm:et} is tight, although there is still room for a three-player version of the equilibrium-transfer theorem since I failed to find a counterexample in the most general case.

\subsection{Order and set-theoretic limitations}\label{sect:lim-order-set}

Proposition~\ref{prop:lim-uncountable} below shows that the countability condition of Theorem~\ref{thm:et} is difficult to weaken in general. 

\begin{proposition}\label{prop:lim-uncountable}
There exists a game satisfying the following:
\begin{itemize}
\item player $1$ has countably many strategies,
\item the preference of player $1$ has no infinite ascending chain,
\item the preference of player $2$ has one maximum and the other outcomes are minimal,
\item the underlying game structure is determined,
\item the game has no Nash equilibrium.
\end{itemize}
\end{proposition}

\begin{proof}
Let $I$ (resp. $C$) be the infinite (resp. cofinite) subsets of the naturals. Consider the two-player game structure $\langle \{1,2\}, (C\cup\{\alpha,\beta\})\times(I\cup\{\alpha,\beta\}),\mathbb{N}\cup\{a,b\},v\rangle$ where the unions are disjoints and where $v$ is as below and $min$ refers to the usual order over $\mathbb{N}$. 
\[\begin{array}{cllll}
v: & (C\cup\{\alpha,\beta\})\times(I\cup\{\alpha,\beta\}) &\to& \mathbb{N}\cup\{a,b\}\\
  & (X,Y)&\mapsto & a				&\mbox{ if } X=Y\in\{\alpha,\beta\}\\
  		&&& b				&\mbox{ if } \{X\}\cup\{Y\}=\{\alpha,\beta\}\\
		&&& min(X)				&\mbox{ if } (X,Y)\in C\times\{\alpha,\beta\}\\
  		&&& min(Y) 				&\mbox{ if } (X,Y)\in\{\alpha,\beta\}\times I\\
		&&& min(X\cap Y-\{min\, X\cap Y\})	&\mbox{ if } X,Y\notin\{\alpha,\beta\}\\
\end{array}\]

Let us first show the determinacy of the game structure. Let $P\subseteq \mathbb{N}\cup\{a,b\}$. If $P\cap\mathbb{N}$ is cofinite, player $1$ can enforce it (by playing it) since $v(X,Y)$ is in $X$ for all $X\in C$ and $Y\in I\cup\{\alpha,\beta\}$, by definition of $v$; so a fortiori player $1$ can enforce $P$. If $P\cap\mathbb{N}$ is not cofinite, $\mathbb{N}\backslash P$ is infinite, so player $2$ can enforce it, and a fortiori  $(\mathbb{N}\cup\{a,b\})\backslash P$.

Second, let us define some preferences: for all $x\in\mathbb{N}\cup\{a\}$, set $x\prec_2b$. Set $b\prec_1 a$ and $n+k+1\prec_1n\prec_1a$ for all $n,k\in\mathbb{N}$. Let us now show that the game has no Nash equilibrium.

\begin{itemize}
\item $v(\alpha,\alpha)<_2v(\alpha,\beta)$ and $v(\beta,\beta)<_2v(\beta,\alpha)$.
\item $v(\alpha,\beta)<_1v(\beta,\beta)$ and $v(\beta,\alpha)<_1v(\alpha,\alpha)$.
\item If $X\in C$ and $Y\in\{\alpha,\beta\}$ then $v(X,Y)<_1v(Y,Y)=a$.
\item If $X\in\{\alpha,\beta\}$ and $Y\in I$ then $v(X,Y)<_2v(X,X')=b$, where $\{X,X'\}=\{\alpha,\beta\}$.
\item If $X,Y\notin\{\alpha,\beta\}$ then $v(X,Y)<_1v(\alpha,Y)$ since $min(X\cap Y-\{min X\cap Y\})<min(Y)$.
\end{itemize}
\end{proof}

Note that replacing "player $1$ has countably many strategies" with "player $2$ has countably many strategies" in Proposition~\ref{prop:lim-uncountable} yields a correct statement. In the proof, indeed, it suffices to swap the sets $I$ and $C$ in the definition of the witness game.

Proposition~\ref{prop:lim-bounded-height} below shows that the (uniform) finite height condition of Theorem~\ref{thm:et} is difficult to weaken in general. The proof idea is similar to that of Proposition~\ref{prop:lim-uncountable}, albeit slightly more complex.

\begin{proposition}\label{prop:lim-bounded-height}
There exists a game satisfying the following:
\begin{itemize}
\item the preference of player $1$ has no infinite chain,
\item the preference of player $2$ has one maximum and the other outcomes are minimal,
\item the underlying game structure is determined,
\item the game has no Nash equilibrium.
\end{itemize}
\end{proposition}

\begin{proof}
For every $n\in\mathbb{N}$ let $A_n:=\{k\in\mathbb{N}\,\mid\,n^2\leq k<(n+1)^2\}$, and let $I:=\{X\subseteq\mathbb{N}\,\mid\,\forall n\in\mathbb{N},\exists m\in\mathbb{N},n\leq|X\cap A_m|\}$, and let $C:=\{X\subseteq\mathbb{N}\,\mid\,\exists n\in\mathbb{N},\forall m\in\mathbb{N},|A_m\backslash X|\leq n\}$. Consider the two-player game structure that is defined by its outcome function below, where $S_1:=C\cup\{\alpha_0,\alpha_1,\dots\}$ and $S_2:=I\cup\{\alpha_0,\alpha_1,\dots\}$, where the unions are disjoints, and where $min$ refers to the usual order over $\mathbb{N}$. . 
\[\begin{array}{cllll}
v: & S_1\times S_2 &\to& \mathbb{N}\cup\{a,b\}\\
  & (X,Y)&\mapsto & a				&\mbox{ if } X=Y=\alpha_n\\
  		&&& b				&\mbox{ if } (X,Y)=(\alpha_i,\alpha_j) \mbox{ with } i\neq j\\
		&&& min(X)				&\mbox{ if } X\in C \mbox{ and } Y=\alpha_n\\
  		&&& min(Y\cap \{n^2,n^2+1,\dots\}) 	&\mbox{ if } X=\alpha_n \mbox{ and } Y\in I\\
		&&& min(X\cap Y\cap A_n-\{min\, X\cap Y\cap A_n\})	&\mbox{ if } (X,Y)\in C\times I,\mbox{ where }\\
		&&&							&n:=min\{k\,\mid\,2\leq|X\cap Y\cap A_k|\}
\end{array}\]

Note that in the last line of the definition above, $\{k\,\mid\,2\leq|X\cap Y\cap A_k|\}$ is non-empty by definition of $C$ and $I$.

Let us first show the determinacy of the game structure. Let $P\subseteq \mathbb{N}\cup\{a,b\}$. If $P\cap\mathbb{N}$ is in $C$, player $1$ can enforce it (by playing it) since $v(X,Y)$ is in $X$ for all $X\in C$ and $Y\in I\cup\{\alpha_0,\alpha_1,\dots\}$, by definition of $v$; so a fortiori player $1$ can enforce $P$. If $P\cap\mathbb{N}$ is not in $C$, $\mathbb{N}\backslash P$ is in $I$, by definition of $C$ and $I$, so player $2$ can enforce it, and a fortiori  $(\mathbb{N}\cup\{a,b\})\backslash P$.

Second, let us define some preferences: set $x\prec_2 b$ for all $x\in\mathbb{N}\cup\{a\}$; set $x\prec_1 a$ for all $x\in\mathbb{N}\cup\{b\}$ and $i\prec_1j$ whenever $n^2\leq j<i<(n+1)^2$ for $i,j,n\in\mathbb{N}$. The game thus defined has no Nash equilibrium, as shown below.
\begin{itemize}
\item $v(\alpha_n,\alpha_n)<_2v(\alpha_n,\alpha_{n+1})$ and $v(\alpha_i,\alpha_j)<_1v(\alpha_j,\alpha_j)$ for $i\neq j$.
\item If $X\in C$ and $Y=\alpha_n$ then $v(X,Y)<_1v(Y,Y)=a$.
\item If $X=\alpha_n$ and $Y\in I$ then $v(X,Y)<_2v(X,\alpha_{n+1})=b$.
\item If $(X,Y)\in C\times I$ then $v(X,Y)<_1v(\alpha_n,Y)$, where $n:=min\{k\,\mid\,2\leq|X\cap Y\cap A_k|\}$.
\end{itemize}
\end{proof}

\subsection{Three-player limitations}\label{sect:3player}

The equilibrium-transfer theorem considers two-player games only, which raises the issue of the existence of a three-player version of the theorem. However, the condition of application of such a version can no longer state a mere determinacy of the game structure, because the very notion of determinacy makes little sense for three players. Instead, the condition may require that some specific games that are derived from and are simpler than the original game all have Nash equilibria. (Note that in the two-player case, the determinacy of the game structure falls into this type of condition.) The counter-examples in this section suggest that there is no general three-player version, without giving a definitive answer, though.

Theorem~\ref{thm:et} can be rephrased as follows: if equipping a given two-player game structure with simple preferences (\textit{i.e.} free of three-outcome chains) always yields a game with a Nash equilibrium, so does equipping the same structure with more complex preferences. The example below shows that one cannot just replace "two-player" with "three-player" in the above statement, since preferences with three-outcome chains may already be problematic.

\begin{remark}\label{prop:3aet-simple}
Let $a$, $b$ and $c$ be three players, let $l$ and $r$ be two strategies available to each player, let $x$, $y$ and $z$ be three possible outcomes, let us define three transitive preferences by $z<_ay<_ax$ and $x<_by<_bz$ and $<_c:=<_b$, and define an outcome function as follows: $v(l,l,l):=v(l,r,l):=v(r,r,l):=y$ and $v(r,l,l):=v(l,l,r):=v(l,r,r):=z$ and $v(r,l,r):=v(r,r,r):=x$. See the graphical representation below, where players $a$, $b$, and $c$ choosing $l$ yields top rows, left columns, and left array, respectively. 

\[\begin{array}{c@{\hspace{1cm}}c}
\begin{array}{|c|c|}
\hline y & y\\
\hline z & y\\
\hline
\end{array}
&
\begin{array}{|c|c|}
\hline z & z\\
\hline x & x\\
\hline
\end{array}
\end{array}
\]
\begin{itemize}
\item The game $\langle\{a,b,c\},(\{l,r\})_{d\in\{a,b,c\}},\{x,y,z\},v,(<_d)_{d\in\{a,b,c\}}\rangle$ has no Nash equilibrium.
\item Let $\{0,1\}^3$ be an alternative set of outcomes, for $i,j,k,n\in\{0,1\}$ let $(0,i,j)\prec_a(1,k,n)$ and $(i,0,j)\prec_b(k,1,n)$ and $(i,j,0)\prec_c(k,n,1)$. Then for all $wl:\{x,y,z\}\to\{0,1\}^3$ the game\\
\noindent $\langle\{a,b,c\},(\{l,r\})_{d\in\{a,b,c\}},\{0,1\}^3,wl\circ v,(\prec_d)_{d\in\{a,b,c\}}\rangle$ has a Nash equilibrium.
\end{itemize}
\end{remark}

\begin{proof}
The original game has no Nash equilibrium, by construction. Let $wl:\{x,y,z\}\to\{0,1\}^3$ and from now on let us consider the modified game only, assume that it has no Nash equilibrium, and draw a contradiction. Both players $a$ and $b$ are stable \textit{w.r.t.} the strategy profile $(l,r,l)$ since $v(l,l,l)=v(l,r,l)=v(r,r,l)=y$ by construction, so $wl(y)=wl\circ v(l,r,l)\prec_cwl\circ v(l,r,r)=wl(z)$. So player $c$ is stable \textit{w.r.t.} the strategy profile $(l,r,r)$, and so is player $b$ since $v(l,l,r)=v(l,r,r)=z$, therefore $wl(z)=wl\circ v(l,r,r)\prec_a wl\circ v(r,r,r)=wl(x)$. So player $a$ is stable \textit{w.r.t.} the strategy profile $(r,r,r)$, and so is player $b$ since $v(r,l,r)=v(r,r,r)=x$, so $wl(x)=wl\circ v(r,r,r)\prec_c wl\circ v(r,r,l)=wl(y)$. Therefore $wl(x)\prec_c wl(y)\prec_c wl(z)$, contradiction since $\prec_c$ has no three-outcome chain.
\end{proof}

A weaker three-player version of Theorem~\ref{thm:et} might obtain by using a large-enough natural number $n$ as follows: "if equipping a given three-player game structure with preferences free of $(n+1)$-outcome chains always yields a game with a Nash equilibrium, so does equipping the same structure with more complex preferences." Proposition~\ref{prop:3aet} below shows that preferences with $(n+1)$-outcome chains may already be problematic. (Note that the case $n=2$ corresponds to Remark~\ref{prop:3aet-simple}.)

\begin{proposition}\label{prop:3aet}
For every natural $2\leq n$ there exists a finite three-player game in normal form that complies with the following:
\begin{itemize}
\item The preferences are linear orders over $\{0,1,\dots,n\}$.
\item The game has no Nash equilibrium.
\item Replacing the preferences with preferences that have no chain of length $n+1$ yields a game with a Nash equilibrium.
\end{itemize}
\end{proposition}

\begin{proof}
Let $2\leq n$ be a natural, let $<_b$ and $<_c$ be the restrictions of the usual order over the naturals to $\{0,\dots,n\}$, and let $<_a:=<_b^{-1}$, that is, $n<_an-1<_a\dots<_a1<_a0$. Let us define the outcome function $v:\{1,\dots,n\}^3\to\{0,\dots,n\}$ below.
\begin{itemize}
\item For $1\leq i\leq n$ let $v(n,i,n):=0$.
\item For $1\leq i<n$ let $v(i,i,n):=i+1$.
\item Otherwise let $v(\cdot,\cdot,n)$ return $n$.
\item For $1\leq i<n$ let $v(n,1,i):=n$.
\item For $1<i<n$ and $1\leq j\leq n$ let $v(i,j,i):=v(j,i,i):=i$.
\item Otherwise let $v$ return $1$.
\end{itemize}

For example, the game $\langle\{a,b,c\},(\{1,2,3,4\})_{d\in\{a,b,c\}},\{0,1,2,3,4\},v,(<_d)_{d\in\{a,b,c\}}\rangle$ is represented below, where player $a$ chooses the row, $b$ the column, and $c$ the array. 

\[\begin{array}{c@{\hspace{1cm}}c@{\hspace{1cm}}c@{\hspace{1cm}}c}
\begin{array}{|c|c|c|c|}
\hline 1 & 1 & 1 & 1\\
\hline 1 & 1 & 1 & 1\\
\hline 1 & 1 & 1 & 1\\
\hline 4 & 1 & 1 & 1\\
\hline
\end{array}
&
\begin{array}{|c|c|c|c|}
\hline 1 & 2 & 1 & 1\\
\hline 2 & 2 & 2 & 2\\
\hline 1 & 2 & 1 & 1\\
\hline 4 & 2 & 1 & 1\\
\hline
\end{array}
&
\begin{array}{|c|c|c|c|}
\hline 1 & 1 & 3 & 1\\
\hline 1 & 1 & 3 & 1\\
\hline 3 & 3 & 3 & 3\\
\hline 4 & 1 & 3 & 1\\
\hline
\end{array}
&
\begin{array}{|c|c|c|c|}
\hline 2 & 4 & 4 & 4\\
\hline 4 & 3 & 4 & 4\\
\hline 4 & 4 & 4 & 4\\
\hline 0 & 0 & 0 & 0\\
\hline
\end{array}
\end{array}
\]

Let us show that the game $\langle\{a,b,c\},(\{1,\dots,n\})_{d\in\{a,b,c\}},\{0,\dots,n\},v,(<_d)_{d\in\{a,b,c\}}\rangle$ witnesses the claim. First, the preferences are linear orders indeed. Second, let us show that there is no Nash equilibrium by case-splitting below. 

\begin{itemize}
\item If $i,k\neq n$, then $v(i,j,k)<_cv(i,j,n)$.  
\item If $i\neq n$, then $v(i,j,n)<_a0=v(n,j,n)$.
\item $v(n,j,n)=0<_cv(n,j,1)$.
\item $v(n,1,1)=n\prec_a 1=v(1,1,1)$ and if $j\neq 1$, then $v(n,j,1)=1\prec_b n=v(n,1,1)$.
\item If $j\neq 1$ and $1<k<n$, then $v(n,j,k)<_bn=v(n,1,k)$.
\item If $1<k<n$, then $v(n,1,k)=n<_a1=v(1,1,k)$.
\end{itemize}

Third, by contraposition, let $\prec_a$, $\prec_b$ and $\prec_c$ be arbitrary acyclic preferences, assume that the game\\
$\langle\{a,b,c\},(\{1,\dots,n\})_{d\in\{a,b,c\}},\{0,\dots,n\},v,(\prec_d)_{d\in\{a,b,c\}}\rangle$ has no Nash equilibrium, and let us prove that $\prec_c=<_c$, thus contradicting the assumption on the chains. If $i\neq n$, by construction $v(i,i,i)=i=v(j,i,i)=v(i,j,i)$ and $v(i,i,j)\in\{1,i,i+1\}$. By assumption there is no Nash equilibrium and the preferences are acyclic, so $i\prec_c1$ or $i\prec_ci+1$ if $1\leq i< n$, so $1\prec_c 2$, and it is provable by induction that $i\prec_ci+1$ for all $1\leq i< n$. Now it suffices to prove $0\prec_c 1$ to conclude. By assumption the strategy profile $(n-1,n,n)$ is not a Nash equilibrium, so $v(n-1,n,n)=n\prec_a 0=v(n,n,n)$ since $v(i,n,n)=v(n-1,j,n)=n$ for $i\neq n$, since $v(n-1,n,k)\in\{1,n-1,n\}$, and by assumption of acyclic preferences. Now, the profile $(n,n,n)$ is not a Nash equilibrium either. Since $v(n,j,n)=0$, since $v(i,n,n)=n$ if $i\neq n$, and since $n\prec_a 0$, the players $a$ and $b$ are sable. Since $v(n,n,k)\in\{0,1\}$, we must have $v(n,n,n)=0\prec_c 1$.
\end{proof}

Note that in the statement of Proposition~\ref{prop:3aet}, one may also modify the games without Nash equilibrium so that they are zero-sum, by defining $z(n):=(-2n,n,n)$ and using the outcome function $z\circ v$ instead of $v$.

Let us now try to find an alternative equilibrium-transfer theorem that still states existence of Nash equilibrium in three-player game. As a lead, notice that many two-player game structures may be derived from a given three-player game structure: first, by slicing the game cuboid, \textit{i.e.} by fixing the strategy of one player; second, by merging any two players into a super player. We may hope that if all thus-derived two-player game structures are determined (and therefore also have Nash equilibria for more complex preferences), then the original three-player game structure has a Nash equilibrium when equipped with simple preferences. However, Proposition~\ref{prop:3sm1} contradicts it.

\begin{proposition}\label{prop:3sm1}
There exists a finite game structure $G=\langle\{a,b,c\},S_a,S_b,S_c,\{X,Y,Z\},v\rangle$ that satisfies the following:
\begin{itemize}
\item for all $s_c\in S_c$ the game structure $\langle\{a,b\},S_a,S_b,\{X,Y,Z\},v(\cdot,\cdot,s_c)\rangle$, which is obtained by slicing $G$ along $s_c$, is determined. (And similarly by slicing along some $s_a\in S_a$ or $s_b\in S_b$.)
\item the game structure $\langle\{a\times b,c\},S_a\times S_b,S_c,\{X,Y,Z\},v')\rangle$, where $v'((s_a,s_b),s_c):=v(s_a,s_b,s_c)$, which is obtained by merging players $a$ and $b$, is determined. (And similarly by merging $c$ with $a$ or $b$.)
\item instantiating $G$ with $X:=(1,0,0)$, $Y:=(0,1,0)$, and $Z:=(0,0,1)$ yields a game without Nash equilibrium.
\end{itemize}
\end{proposition}

\begin{proof}
Let us define a three-player game structure, each player having six strategies. Informally, the $6\times 6\times 6$ empty cube is filled with three $3\times 3\times 6$ cylinders and two $3\times 3\times 3$ cylinders. The cross section of each of these cylinders looks like this:

\[\begin{array}{|c|c|c|}
\hline X & Y & Z\\
\hline Y & Z & X\\
\hline Z & X & Y\\
\hline
\end{array}\]

Let us define the outcome function of the game formally below, where each of the five blocks of three lines corresponds to a cylinder.

\[\begin{array}{l@{\hspace{2cm}}l}
v(1,1,\cdot):=v(2,3,\cdot):=v(3,2,\cdot):=X & v(\cdot,4,1):=v(\cdot,5,3):=v(\cdot,6,2):=X\\
v(1,2,\cdot):=v(2,1,\cdot):=v(3,3,\cdot):=Y & v(\cdot,4,2):=v(\cdot,5,1):=v(\cdot,6,3):=Y\\
v(1,3,\cdot):=v(2,2,\cdot):=v(3,1,\cdot):=Z & v(\cdot,4,3):=v(\cdot,5,2):=v(\cdot,6,1):=Z\\
\\
v(4,\cdot,4):=v(5,\cdot,6):=v(6,\cdot,5):=X\\
v(4,\cdot,5):=v(5,\cdot,4):=v(6,\cdot,6):=Y\\
v(4,\cdot,6):=v(5,\cdot,5):=v(6,\cdot,4):=Z\\
\end{array}\]

For $1\leq i\leq 3$, set the following:

\[\begin{array}{l@{\hspace{2cm}}l}
v(4,1,i):=v(5,3,i):=v(6,2,i):=X & v(4,i,4):=v(5,i,6):=v(6,i,5):=X\\
v(4,2,i):=v(5,1,i):=v(6,3,i):=Y & v(4,i,5):=v(5,i,4):=v(6,i,6):=Y\\
v(4,3,i):=v(5,2,i):=v(6,1,i):=Z & v(4,i,6):=v(5,i,5):=v(6,i,4):=Z\\
\end{array}\]

A typical section of the whole game structure, actually $v(\cdot,\cdot,1)$, looks as below: the cross sections of two cylinders on the left-hand side, the "outer face" of another cylinder on the right-hand side:

\[\begin{array}{|c|c|c|c|c|c|}
\hline X & Y & Z & X & Y & Z\\
\hline Y & Z & X& X & Y & Z\\
\hline Z & X & Y& X & Y & Z\\
\hline X & Y & Z & X & Y & Z\\
\hline Y & Z & X& X & Y & Z\\
\hline Z & X & Y& X & Y & Z\\
\hline
\end{array}\]

Instantiating the game structure with $X:=(1,0,0)$, $Y:=(0,1,0)$, and $Z:=(0,0,1)$ yields a game without Nash equilibrium, because none of the cylinders constituting the whole game has a Nash equilibrium.

Nonetheless, all the two-player game structures that are derived from the three-player game structure by slicing or merging have a Nash equilibrium, as partially justified below.

Slicing: fix the strategy of c between $1$ and $3$ (resp. $4$ and $6$), then player $b$ (resp. $a$) can choose the outcome he/she wants. For instance fixing the strategy of $c$ to $1$ yields the game represented graphically above. Similar situations arise when fixing strategies of $a$ or $b$. 

Merging: The three $3\times 3\times 6$ cylinders ensure that any two players can collectively enforce any outcome. 
\end{proof}

The next and last idea is to combine the two previous attempts to get even stronger assumptions: given a game structure, we may hope that if slicing or merging it always yields determined game structures, and if equipping it with simple preferences always yields games with Nash equilibrium, then so does equipping it with more complex preferences. However, the following does not sound very promising.

\begin{proposition}\label{prop:3sm}
There exists a finite game structure $G=\langle\{a,b,c\},S_a,S_b,S_c,\{X,Y,Z\},v\rangle$ that satisfies the following:
\begin{itemize}
\item for all $s_c\in S_c$ the game structure $\langle\{a,b\},S_a,S_b,\{X,Y,Z\},v(\cdot,\cdot,s_c)\rangle$, which is obtained by slicing $G$ along $s_c$, is determined. (And similarly by slicing along some $s_a\in S_a$ or $s_b\in S_b$.)
\item the game structure $\langle\{a\times b,c\},S_a\times S_b,S_c,\{X,Y,Z\},v')\rangle$, where $v'((s_a,s_b),s_c):=v(s_a,s_b,s_c)$, which is obtained by merging players $a$ and $b$, is determined. (And similarly by merging $c$ with $a$ or $b$.)
\item Equipping $G$ with preferences that are free of three-outcome chains yields a game a Nash equilibrium.
\item Equipping $G$ with $Z<_aY<_aX$ and $X<_bZ<_bY$ and $Y<_cX<_cZ$ yields a game without Nash equilibrium.
\end{itemize}
\end{proposition}

\begin{proof}
Let us define a three-player game structure that is not as symmetric as the one from Proposition~\ref{prop:3sm1}: Player $a$ has four strategies and players $b$ and $c$ have seven strategies each. Informally, the $4\times 7\times 7$ empty cuboid is filled with one $4\times 3\times 3$ cylinder like in Proposition~\ref{prop:3sm1}, four thinner $2\times 2\times 7$ cylinders, and four shorter versions of these, namely $2\times 2\times 3$ cylinders. The different sections of the thinner cylinders look like these:

\[\begin{array}{ccc}
\begin{array}{|c|c|}
\hline X & Y \\
\hline Y & X \\
\hline
\end{array}
&
\begin{array}{|c|c|}
\hline X & Z \\
\hline Z & X \\
\hline
\end{array}
&
\begin{array}{|c|c|}
\hline Y & Z \\
\hline Z & Y \\
\hline
\end{array}
\end{array}
\]

Let us define the outcome function of the game formally below, where each of the five blocks of three lines corresponds to a cylinder.

\[\begin{array}{l}
v(\cdot,5,1):=v(\cdot,6,3):=v(\cdot,7,2):=X\\
v(\cdot,5,2):=v(\cdot,6,1):=v(\cdot,7,3):=Y\\
v(\cdot,5,3):=v(\cdot,6,2):=v(\cdot,7,1):=Z\\
\\
v(1,1,\cdot):=v(2,2,\cdot):=v(1,3,\cdot):=v(2,4,\cdot):=X\\
v(1,2,\cdot):=v(2,1,\cdot):=Y\\
v(1,4,\cdot):=v(2,3,\cdot):=Z\\
\\
v(3,\cdot,5):=v(4,\cdot,4):=X\\
v(3,\cdot,7):=v(4,\cdot,6):=Y\\
v(3,\cdot,4):=v(4,\cdot,5):=v(3,\cdot,6):=v(4,\cdot,7):=Z\\
\end{array}\]

For $1\leq i\leq 3$, set the following:

\[\begin{array}{l}
v(3,1,i):=v(4,2,i):=v(3,3,i):=v(4,4,i):=X \\ 
v(3,2,i):=v(4,1,i):=Y \\
v(3,4,i):=v(4,3,i):=Z \\
\\
v(1,i+4,5):=v(2,i+4,4):=X\\
v(1,i+4,7):=v(2,i+4,6):=Y\\
v(1,i+4,4):=v(2,i+4,5):=v(1,i+4,6):=v(2,i+4,7):=Z\\
\end{array}\]

Typical cross sections of the whole game structure look like the two below, $v(\cdot,\cdot,1)$ on the left-hand side and $v(\cdot,\cdot,7)$ on the right-hand side.

\[\begin{array}{c@{\hspace{2cm}}c}
\begin{array}{|c|c|c|c|c|c|c|}
\hline X & Y & X & Z & X &Y & Z\\
\hline Y & X & Z & X & X & Y & Z\\
\hline X & Y & X & Z & X & Y & Z\\
\hline Y & X & Z & X & X & Y & Z\\
\hline
\end{array}
&
\begin{array}{|c|c|c|c|c|c|c|}
\hline X & Y & X & Z & Y & Y & Y\\
\hline Y & X & Z & X & Z & Z & Z\\
\hline Y & Y & Y & Y & Y & Y & Y\\
\hline Z & Z & Z & Z & Z & Z & Z\\
\hline
\end{array}
\end{array}\]

Equipping the game structure with the preferences $Z<_aY<_aX$ and $X,Z<_bY$ and $X,Y<_cZ$ yields a game without Nash equilibrium.

Nonetheless, slicing the game structure or merging any two of its players yields a determined game structure. It is proved either by similar arguments as in Proposition~\ref{prop:3sm1}, or as follows to show that the cross section $v(\cdot,\cdot,7)$ above to the right is determined: if $Y$ or $Z$ makes player $a$ win, player $a$ wins for sure by playing one of the last two rows; if $Y$ and $Z$ make player $a$ lose, player $b$ wins for sure by playing the last column.

Let us now show that equipping the game structure with preferences that are free of three-outcome chains yields a game with a Nash equilibrium: If two players share a preferred outcome, they can collectively enforce it, which yields a Nash equilibrium, so now let us assume that the three players prefer distinct outcomes. The following array points to one Nash equilibrium for each of the six permutations of $(X,Y,Z)$ as preferred outcomes: 

\[\begin{array}{|c|c|c|c|}
	\cline{1-4}
	a & b & c & \mbox{Nash equilibrium}\\
	\cline{1-4}
	Z & X & Y & 1,1,1\\
	\cline{1-4}
	Z & Y & X  &  1,2,1\\
	\cline{1-4}
	Y & X & Z  &  1,3,1\\
	\cline{1-4}
	Y & Z & X  &  1,4,1\\
	\cline{1-4}
	X & Y & Z  & 4,7,7\\
	\cline{1-4}
	X & Z & Y  &  4,7,6\\
	\cline{1-4}
\end{array}\]
\end{proof}

Proposition~\ref{prop:3sm} does not fully dash the hope for a three-player version of the equilibrium-transfer theorem, though. Indeed, an alternative, even weaker statement could be as follows for a given natural number $n$: "If merging or slicing a given three-player game always yields determined structures, and if replacing the preferences of the game with preferences of height at most $n$ always yields games with a Nash equilibrium, then the original game also has a Nash equilibrium." The case $n=2$ is disproved by Proposition~\ref{prop:3sm}, but the cases $3\leq n$ are still open. If there are counterexamples too, building them and proving their combinatorial property might be rather complex, though.

\section{Conclusion}\label{sect:c}

This article has shown that every determinacy result over a given two-player game structure is transferable into existence of multi-outcome Nash equilibrium over the same game structure. Moreover, when the outcomes are finitely many, the proof provides an algorithm that computes a Nash equilibrium without significant complexity loss compare to the win-lose case.

Contrary to most game-theoretic results, which state that every game of a given class of games has some property, this result is a higher-order theorem: it states that every class of games that is derived from any game structure has itself some property.

If the heights of the preferences of the two players are finite, the equilibrium transfer holds regardless of the game structure; furthermore, if the structure has countably many strategies, the finite-height condition can be relaxed and phrased as absence of infinite ascending sequences, which was the hardest to prove in this article. 

Although counterexamples from Section~\ref{sect:lim-order-set} show that these conditions are useful, there is still room for fine-tuning. In particular, it is still open whether the following is a sufficient condition for equilibrium transfer: "the strategy set of one player is countable and the preferences of the players have no infinite sequences (even descending)".

Section~\ref{sect:app-et} gave three examples of applications of the equilibrium-transfer theorem. Apart from the generalisation of positional determinacy of parity games, which is new, the two other obtained results are weaker than existing results; but the key point here is, however, that the three applications are almost effortless, and that the same would hold for further applications.

The two above-mentioned existing generalisations of Borel determinacy and finite-memory determinacy of Muller games are indeed more general than what can be obtained by application of the equilibrium-transfer theorem, because they hold in a multi-player setting. The proofs are ad hoc, though, which raises the question of a uniform equilibrium-transfer theorem for games with three or more players. A natural attempt is to replace the determinacy condition with existence of Nash equilibrium in some simpler derived games. Counterexamples in Section~\ref{sect:3player} disproved some simple variants of such an attempt, but existence of a slightly more complex variant is still an open question. Alternatively, one may try to add strong conditions on the structure, \textit{e.g.}, after noticing that the players play sequentially in both Muller games and the games used for Borel determinacy. 

\medskip

Unexpectedly, the finite-height condition of Theorem~\ref{thm:et} leads to an interesting general phenomenon about preferences in game theory: contrary to a widespread belief, linear orders do not account for partial orders. Indeed the remark below considers two finite-height preferences, therefore fulfilling Condition~\ref{cond:thm-et2} of Theorem~\ref{thm:et}, yet for all possible linear extensions of these preferences, equilibrium transfer does not hold!

\begin{remark}\label{rem:total}
Let $\prec_1$ and $\prec_2$ be two binary relations over $\mathbb{N}$ that are defined by $2n\prec_1 2n+1$ and $\prec_2:=\prec_1^{-1}$. For all $<_1$ and $<_2$ linear extensions of $\prec_1$ and $\prec_2$ respectively, there exists a game satisfying Condition~\ref{cond:thm-et1} of Theorem~\ref{thm:et}, but without Nash equilibrium.
\end{remark}

\begin{proof}
If the inverse of $<_1$ is not a well-order, the game $\langle \{1\},\mathbb{N},\mathbb{N},id,\{<_1\}\rangle$ has no Nash equilibrium although the induced structure is determined, so let us assume that the inverse of $<_1$ is a well-order. Since the sequence $(2n)_{n\in\mathbb{N}}$ has no $<_1$-increasing subsequence, it has a $<_1$-decreasing subsequence $(2\phi(n))_{n\in\mathbb{N}}$ (as a consequence of Ramsey Theorem). Setting $a:=2\phi(0)+1$ and $b:=2\phi(0)$ and $x_n=2\phi(n+1)$ embeds the preferences from Proposition~\ref{prop:lim-uncountable} into $<_1$ and $<_2$, respectively, so the witness game from Proposition~\ref{prop:lim-uncountable} also witnesses the remark at hand.
\end{proof}

Nonetheless, it is often very convenient to consider linearly ordered preferences only, when actually done without loss of generality. Remark~\ref{rem:total} above just exemplifies that one ought to be very cautious because a loss of generality may actually occur.

\section*{Acknowledgement}
I thank Achim Blumensath and Michael Ummels for discussions on parity games and the like, Alexander Kreuzer for explanations on Ramsey's theorem, an anonymous referee in particular for a remark that helped shorten and clarify the proof of Theorem~\ref{thm:et}, and I am especially grateful to Vassilios Gregoriades for useful discussions and advice, and his determined and determinative help with Proposition~\ref{prop:lim-uncountable}.

\bibliographystyle{plain}
\bibliography{article}

%%
%% Bibliography
%%

%% Either use bibtex (recommended), but commented out in this sample

%\bibliography{dummybib}

%% .. or use bibitems explicitely

\end{document}